\documentclass{mfat}
\pagespan{184}{196}

\usepackage{mmap}
\usepackage[utf8]{inputenc}

\theoremstyle{plain}
\newtheorem{theorem}{Theorem}
\newtheorem{lemma}{Lemma}
\newtheorem{corollary}{Corollary}
\theoremstyle{definition}

\allowdisplaybreaks

\begin{document}

\title[Level sets of asymptotic mean of digits function]
    {Level sets of asymptotic mean of digits function for $4$-adic representation of real number}

\author{M. V. Pratsiovytyi}
\address{National Pedagogical Dragomanov University, 9 Pirogova, Kyiv, 01601, Ukraine;
Institute of Mathematics, National Academy of Sciences of Ukraine,
3 Tereshchenkivs'ka, Kyiv, 01601, Ukraine}
\email{prats4444@gmail.com}

\author{S. O. Klymchuk}
\address{Institute of Mathematics, National Academy of Sciences of Ukraine,
3 Tereshchenkivs'ka, Kyiv, 01601, Ukraine}
\email{svetaklymchuk@gmail.com}

\author{O. P. Makarchuk}
\address{Institute of Mathematics, National Academy of Sciences of Ukraine,
3 Tereshchenkivs'ka, Kyiv, 01601, Ukraine}
\email{makolpet@gmail.com}

\date{04/09/2014;\ \  Revised 20/10/2015}

\begin{abstract}
  We study topological, metric and fractal properties of the level
  sets
$$S_{\theta}=\{x:r(x)=\theta\}$$ of the function $r$ of asymptotic
mean of digits of a number $x\in[0;1]$ in its $4$-adic representation,
$$r(x)=\lim\limits_{n\to\infty}\frac{1}{n}\sum\limits^{n}_{i=1}\alpha_i(x)$$
if the asymptotic frequency $\nu_j(x)$ of at least one digit does not
exist, were $$ \nu_j(x)=\lim_{n\to\infty}n^{-1}\#\{k: \alpha_k(x)=j,
k\leqslant n\}, \:\: j=0,1,2,3. $$
\end{abstract}

\maketitle

\section{Introduction}
Let $2\leqslant s\in N$ and $\mathcal{A}_s=\{0,1,\ldots,s-1\}$ be an
alphabet of $s$-adic number system. By
$\Delta^s_{\alpha_1(x)\alpha_2(x)\ldots\alpha_k(x)\ldots}$ denote the
$s$-adic representation of a number $x\in[0;1]$, i.e.,
$$
 x=\displaystyle\frac{\alpha_1}{s}+
   \displaystyle\frac{\alpha_2}{s^2}+\cdots+
   \displaystyle\frac{\alpha_n}{s^n}+\cdots\equiv\Delta^s_{\alpha_1\alpha_2\ldots\alpha_k\ldots} ,
$$
where $\mathcal{A}_s\ni\alpha_k=\alpha_k(x)$ is the \emph{$k$th $s$-adic digit of the number} $x$.
Note that some numbers  have two $s$-adic representations such that
$$
\Delta^s_{c_1\ldots c_{k-1}c_k(0)}=\Delta^s_{c_1\ldots c_{k-1}[c_k-1](s-1)},
$$ 
where $(i)$ is a period in the number representation. These numbers
are called \emph{$s$-adic--rational}. The rest of numbers have unique
representations and are called \emph{$s$-adic--irra\-tional.}  The $k$th
digit of the number, as its function, is well defined after agreement
to use the first $s$-adic representation only, i.e., the
representation with period $(0)$.

\emph{An asymptotic mean} (or simply \emph{mean}) \emph{of digits of
  the number $x$} is a value $r(x)$ such that $$
\lim\limits_{n\to\infty}\frac{1}{n}\sum\limits^{n}_{i=1}\alpha_i(x)\equiv
r(x), $$ (if the limit exists), where $\mathcal{A}_s\ni\alpha_i$ are
digits of the $s$-adic representation of the number $x\in[0;1]$. The
value $n^{-1}\sum\limits^{n}_{i=1} \alpha_i(x)\equiv r_n(x)$ is called
\emph{relative mean of digits} in the $s$-adic representation of~$x$.

In this paper we study properties of the function $r$ of asymptotic
mean of digits, in particular, topological, metric, and fractal
properties of number sets with a preassigned \emph{asymptotic mean of
  digits}. Namely, we investigate sets $$ S_{\theta}\equiv\left\{
  x:\lim_{n\to\infty}\frac{1}{n}\sum^{n}_{i=1}\alpha_i(x)=\theta\in
  [0; s-1] \right\}, $$ that are level sets of the function $r$ (indeed,
$r^{- 1}(\theta) = S_{\theta}$).  If $\theta \not \in [0;s-1]$ then it
is easily proved that the set $S_{\theta}$ is empty.

Asymptotic mean of digits of a number $x$ is closely related to the
concept of digit frequency of the number.

Let $N_i(x,k)$ be the number the digits ``$i$''$\in\mathcal{A}_s$
appears in the $s$-adic representation
$\Delta^s_{\alpha_1\alpha_2\ldots\alpha_k\ldots}$ of the real number
$x\in[0;1]$ to $k$th place including, i.e.
$$
N_i(x,k)=\# \{j:\, \alpha_j (x)=i, \, j\leqslant k\}.
$$

The \emph{frequency (asymptotic frequency) of a digit ``$i$''} in the $s$-adic
representation of a number $x\in[0;1]$ is the limit (if it exists) such
that
$$
\nu_i(x) =\lim\limits_{k\to\infty} v_i^{(k)},
$$
where $v_i^{(k)}= k^ {- 1} N_i(x, k)$ is called \emph{relative
  frequency of the digit ``$i$''} in the $s$-adic representation of a number
$x$.

The frequency function $\nu_i(x)$ of a digit ``$i$'' in the $s$-adic
representation of a number $x\in[0;1]$ is well defined for
$s$-adic--irrational numbers, and, for $s$-adic--rational numbers, it is
well defined after agreement to use representation with period $(0)$
only.

Different mathematical objects with fractal properties were defined
and studied in terms of frequencies. First of all it is
Besicovitch--Eggleston's sets \cite{Besic2, Egg1} $$
E[\tau_0,\tau_1,\ldots,\tau_{s-1}]=\{x:x=\Delta^s_{\alpha_1\alpha_2\ldots\alpha_k\ldots},\,\nu_i(x)=\tau_i\geqslant0,\,\,i=\overline{0,s-1}\},
$$ the Hausdorff--Besicovitch dimension of the sets is equal to
\cite{Bill} to $$ \alpha_0 (E[\tau_0, \tau_1, \ldots, \tau_{s-1}])
= - \frac{\ln \tau_0^{\tau_0}\tau_1^{\tau_1} \ldots
\tau_{s-1}^{\tau_{s-1}}} {\ln s}. $$

 The number $x$ is called
\emph{normal for basis $s$} if the value $\nu_i(x)$ exists  for all
$i\in\mathcal{A}_s$ and equals to $s^{- 1}$. The set of all normal
for  basis $s$ numbers is  the only Besicovitch--Eggleston's  set
of positive and even full Lebesgue measure.

Normal for every natural basis $s \geqslant 2 $ number $x$ is called
\emph{normal.}  According to the famous Borel's theorem \cite{Bor1} we
see that \emph{Lebesgue measure of the set of normal numbers is equal
  to 1.}

In the papers \cite{AlPrTor}, \cite{PrTorb} it was proved that the
Hausdorff--Besicovitch dimension of abnormal and essentially abnormal
number sets (i.e. number sets having not frequency of at least one
digit or having not frequencies of all digits respectively) is equal
to $1$.

If a number $x$ has all digits frequencies then the relationship between
asymptotic mean of digits and digits frequencies of the number $x$ is the
following: $$ r(x) = \nu_1 (x) +2 \nu_2 (x) + \cdots + (s-1) \nu_{s-1}
(x). $$

When $s = 2$, it is obvious that the asymptotic mean of digits is equal to
the frequency of digit ``1''. So we do not examine this case. The case
$s=3$ was studied in papers \cite{My2,My4}.  It is unique since it is
the only case where the set $S_{\theta}$ is a union of two disjoint sets
$ \Theta_1 $ and $ \Theta_2 $ such that
$$
\begin{array}{ll}
\Theta_1&\equiv\left\{x:\text{frequencies of all digits exist}\right\},\\
\Theta_2&\equiv\left\{x:\text{frequency of any digit does not exist}\right\}.
\end{array}
$$
In cases $s>3$ the set $S_{\theta}$ is a union of tree disjoint sets  such that
$$
\begin{array}{ll}
\!\Theta_1&\!\!\!\!\!\equiv\!\left\{x\!:\text{frequencies of all digits exist}\right\},\\
\!\Theta_2&\!\!\!\!\!\equiv\!\left\{x\!:\text{frequency  of at least one digit exists and of at least one digit does not exist}\right\}\!,\\
\!\Theta_3&\!\!\!\!\!\equiv\!\left\{x\!:\text{frequency of any digit does not exist}\right\}.
\end{array}
$$  

In this paper we study the case $s=4$ since it is the easiest and
modeling in the last class.  Our previous paper \cite{My3} was
devoted to studying properties of the set $\Theta_1$, and this one
deals with the sets $\Theta_2$ and $\Theta_3$.

\section{The object of study}
\begin{lemma}\label{lem1}
  If in the $4$-adic representation of a real number $x\in[0;1]$ the
  frequency of one digit does not exist, then the frequency of at least
  one more digit does not exist.
\end{lemma}
\begin{proof}
  Suppose the frequency $\nu_k(x_0)$ does not exist,
  i.e., $\lim\limits_{n\to\infty}\frac{N_k(x_0,n)}{n}$ does not
  exist. Since
$$
\frac{N_k(x_0,n)}{n}=1-\frac{N_j(x_0,n)}{n}-\frac{N_m(x_0,n)}{n}-\frac{N_l(x_0,n)}{n},
$$
we see
that $$\lim\limits_{n\to\infty}\left(\frac{N_j(x_0,n)}{n}+\frac{N_m(x_0,n)}{n}+\frac{N_l(x_0,n)}{n}\right)$$
does not exist. It means that at least one of the limits
$\lim\limits_{n\to\infty}\frac{N_{j}(x_0,n)}{n}$,
$\lim\limits_{n\to\infty}\frac{N_{m}(x_0,n)}{n}$ or
$\lim\limits_{n\to\infty}\frac{N_l(x_0,n)}{n}$, where
$\{j,k,l,m\}=\{0,1,2,3\}$ do not exist.
\end{proof}

\begin{lemma} \label{lem2} If in the $4$-adic representation of a real
  number $x\in[0;1]$ the asymptotic mean of digits,  $r(x)$, and at least two
  $4$-adic digits frequencies $\nu_i(x)$, $\nu_j(x)$, where
  $i,j\in\{0,1,2,3\}$, exist, then the remaining two $4$-adic digits
  frequencies of the number $x$ exist.
\end{lemma}
\begin{proof}
Consider the system
\begin{equation}
  \label{*}
  \left\{
    \begin{aligned}
      &v^{(n)}_0 + v^{(n)}_1 + v^{(n)}_2 + v^{(n)}_3 = 1,\\
      &v^{(n)}_1 + 2v^{(n)}_2 + 3v^{(n)}_3 = r_n.
    \end{aligned}
  \right.
\end{equation}

Let $i,j\in\{1,2,3\}$. Since
$\lim\limits_{n\to\infty}v^{(n)}_i=\nu_i(x)$,
$\lim\limits_{n\to\infty}r_n=\theta$, we see that from the second
equation of system \eqref{*} it follows that
$\lim\limits_{n\to\infty}v^{(n)}_k$, $k\in\{1,2,3\}\setminus\{i,j\}$,
exists, i.e. the frequency $\nu_k(x)$ exists. Then from the first
equation of system \eqref{*} it follows that $\nu_0(x)$ exists.

Let $i=0$, $j=1$. Then from system \eqref{*} we have
$$\left\{
    \begin{aligned}
&v^{(n)}_3=r_n+v^{(n)}_1-2v^{(n)}_0-2,\\
&v^{(n)}_2=1-v^{(n)}_0-v^{(n)}_1-v^{(n)}_3=3+v^{(n)}_0-2v^{(n)}_1-r_n,
\end{aligned}
  \right. $$
which implies existence of the frequencies $\nu_2(x)$ and $\nu_3(x)$.

Let $i=0$, $j=2$. Then from system \eqref{*} we obtain
$$\left\{
    \begin{aligned}
&v^{(n)}_3=\frac{1}{2}(r_n-v^{(n)}_2+v^{(n)}_0-1),\\
&v^{(n)}_1=1-v^{(n)}_0-v^{(n)}_2-v^{(n)}_3=\frac{3}{2}-\frac{3}{2}v^{(n)}_0-\frac{3}{2}v^{(n)}_2-\frac{r_n}{2},
\end{aligned}
  \right.$$
which implies existence of the frequencies $\nu_1(x)$ and $\nu_3(x)$.

\pagebreak
Let $i=0$, $j=3$. Then from system \eqref{*} we have
$$\left\{
    \begin{aligned}
&v^{(n)}_2=r_n-2v^{(n)}_3+v^{(n)}_0-1,\\
&v^{(n)}_1=1-v^{(n)}_0-v^{(n)}_2-v^{(n)}_3=2-r_n-2v^{(n)}_0+2v^{(n)}_3,
\end{aligned}
  \right.$$
which implies existence of the frequencies $\nu_1(x)$ and $\nu_2(x)$.
\end{proof}

\begin{corollary}{\rm{(from Lemmas 1 and 2)}}.  A number $x\in
  S_{\theta}$ can not have frequencies of only two or of only three
  $4$-adic digits.
\end{corollary}

From Lemma \ref{lem1} it follows that if a number does not have
frequency of at least one $4$-adic digit then it does not have
frequency of one more digit, therefore, the number $x\in S_{\theta}$
can not have frequencies of only three digits. According to Lemma
\ref{lem2}, if a number $x\in S_{\theta}$ has frequencies of at least
two digits then it has frequencies of all digits, therefore, the number $x$
can not have frequencies of only two $4$-adic digits.

Hence, the set $S_\theta$ can be represented as a union of three disjoint sets
$\Theta_1$, $\Theta_2$ and $\Theta_3$ such that
$$
\begin{array}{ll}
\Theta_1&\equiv\left\{x:\nu_i(x)\text{  exist}, \forall i\in\mathcal{A}_4\right\},\\
\Theta_2&\equiv\left\{x:\text{exist frequency of only one 4-adic digit}~\nu_i(x), i\in\mathcal{A}_4\right\},\\
\Theta_3&\equiv\left\{x:\nu_i(x)\text{  do not exist}, \forall i\in\mathcal{A}_4\right\}.
\end{array}
$$

In the following sections we study properties of the sets $\Theta_2$ and
$\Theta_3$.

\section{Abnormal numbers that have asymptotic mean of digits}

\begin{theorem}
If $\theta=0$ or $\theta=3$, then $\Theta_2= \Theta_3=\varnothing$.
\end{theorem}
\begin{proof}
  Let $\theta=0$. If $\lim\limits_{n\to\infty}r_n(x)=0$, then for any
  $i\in\{1,2,3\}$ the following inequality holds: $0\leqslant
  v^{(n)}_i(x)\leqslant
  v^{(n)}_1(x)+2v^{(n)}_2(x)+3v^{(n)}_3(x)=r_n(x)\to 0$, as
  $n\to\infty$. Therefore,
  $\nu_i(x)=\lim\limits_{n\to\infty}v^{(n)}_i(x)=0$ and
  $\nu_0(x)=1$. Hence, $\Theta_2=\Theta_3=\varnothing$.

  Let $\theta=3$. If $\lim\limits_{n\to\infty}r_n(x)=3$, then
  multiplying the first equation of system \eqref{*} by $3$ and subtracting the
  second equation of the system, we obtain that
  $3v^{(n)}_0+2v^{(n)}_1+v^{(n)}_2=3-r_n$. Hence $0\leqslant
  v^{(n)}_i(x)\leqslant
  3v^{(n)}_0(x)+2v^{(n)}_1(x)+v^{(n)}_2(x)=3-r_n(x)\to 0$ as
  $n\to\infty$. Therefore, $\nu_i(x)=0$, for all
  $i\in\{0,1,2\}$. Hence $\nu_3(x)=1$ and $\Theta_2=
  \Theta_3=\varnothing$.
\end{proof}

\begin{lemma} \label{teo4} Let $(s_k)$ be a sequence of positive
  integers and the following conditions hold: $
  \lim\limits_{k\to\infty}s_k=\infty $, $
  \lim\limits_{k\to\infty}\displaystyle\frac{k}{\sum\limits^{k}_{i=1}s_i}=0
  $, $\alpha_1,\alpha_2,\beta_1,\beta_2\geqslant0$,
  $\alpha_1\neq\alpha_2$, $\beta_1\neq\beta_2$.  Then there exist
  sequences $a_n(\alpha_1,\alpha_2)$ and $b_n(\beta_1,\beta_2)$ such
  that $a_n(\alpha_1,\alpha_2)\in\{\alpha_1,\alpha_2\}$ and
  $b_n(\beta_1,\beta_2)\in\{\beta_1, \beta_2\}$ for all $n\in N$ and the
  limits 
$$\lim\limits_{k\to\infty}\frac{\sum\limits^{k}_{i=1}[a_i(\alpha_1,\alpha_2)\cdot
  s_i]} {\sum\limits^{k}_{i=1}s_i}
\quad\text{and}\quad
  \lim\limits_{k\to\infty}\frac{\sum\limits^{k}_{i=1}[b_i(\beta_1,\beta_2)\cdot s_i]}
                               {\sum\limits^{k}_{i=1}s_i}$$ do not exist.
\end{lemma}
\begin{proof}
Without loss of generality let $\alpha_2>\alpha_1$, $\beta_2>\beta_1$, \\
$\displaystyle\frac{\sum\limits^{k}_{i=1}[\lambda s_i]}{\sum\limits^{k}_{i=1}s_i}\leqslant
 \displaystyle\frac{\sum\limits^{k}_{i=1}\lambda s_i}{\sum\limits^{k}_{i=1}s_i}=\lambda$
 \quad and \quad
$\displaystyle\frac{\sum\limits^{k}_{i=1}[\lambda s_i]}{\sum\limits^{k}_{i=1}s_i}>
 \displaystyle\frac{\sum\limits^{k}_{i=1}\lambda(s_i-1)}{\sum\limits^{k}_{i=1}s_i}=
 \lambda-\displaystyle\frac{k}{\sum\limits^{k}_{i=1}s_i}\to\lambda $ $\quad \text{as} \quad k\to\infty.$

Hence, $\lim\limits_{k\to\infty}\displaystyle\frac{\sum\limits^{k}_{i=1}[\lambda s_i]}{\sum\limits^{k}_{i=1}s_i}=\lambda.$

Suppose that $\varepsilon>0$ satisfies
$\alpha_2-\varepsilon>\alpha_1+\varepsilon$ and
$\beta_2-\varepsilon>\beta_1+\varepsilon$. Let $r_1,~l_1$ be smallest
positive integers such that for any $n>r_1$ and $m>l_1$ the following
inequalities hold:
$$\frac{\sum\limits^{n}_{i=1}[\alpha_2 s_i]}{\sum\limits^{n}_{i=1}s_i}>\alpha_2-\varepsilon
\text{ \:\:\: and \:\:\: }
\frac{\sum\limits^{m}_{i=1}[\beta_2 s_i]}{\sum\limits^{n}_{i=1}s_i}>\beta_2-\varepsilon.$$
Denote $n_1=\max(r_1,l_1)$.

Let $r_2,~l_2$ be smallest positive integers such that for all
$r_2>n_1$, $l_2>n_1$ and $n>r_2$, $m>l_2$,
$$\frac{\sum\limits^{n_1}_{i=1}[\alpha_2 s_i]+\sum\limits^{n}_{i=n_1+1}[\alpha_1 s_i]}
       {\sum\limits^{n}_{i=1}s_i}<\alpha_1+\varepsilon
\text{ \:\:\: and \:\:\: }
\frac{\sum\limits^{n_1}_{i=1}[\beta_2 s_i]+\sum\limits^{m}_{i=n_1+1}[\beta_1 s_i]}
     {\sum\limits^{n}_{i=1}s_i}<\beta_1+\varepsilon.$$
Denote $n_2=\max(r_2,l_2)$.

Let $r_3,~l_3$ be smallest positive integers such that for any
$r_3>n_2$, $l_3>n_2$ and $n>r_3$, $m>l_3$, the following inequalities hold:
$$\frac{\sum\limits^{n_1}_{i=1}[\alpha_2 s_i]+\sum\limits^{n_2}_{i=n_1+1}[\alpha_1 s_i]+\sum\limits^{n}_{i=n_2+1}[\alpha_2 s_i]}
       {\sum\limits^{n}_{i=1}s_i}>\alpha_2-\varepsilon
$$
and
$$
\frac{\sum\limits^{n_1}_{i=1}[\beta_2 s_i]+\sum\limits^{n_2}_{i=n_1+1}[\beta_1 s_i]+\sum\limits^{m}_{i=n_2+1}[\beta_2 s_i]}
     {\sum\limits^{n}_{i=1}s_i}>\beta_2-\varepsilon.$$
Denote $n_3=\max(r_3,l_3)$. And so on.

Let $a_n(\alpha_1,\alpha_2)=\alpha_1$ if $n\in\{1,\ldots,n_1-1\}$,
$a_n(\alpha_1,\alpha_2)=\alpha_2$ if $n\in\{n_k,\ldots,n_{k+1}-1\}$
and $k$ is not an even integer; $a_n(\alpha_1,\alpha_2)=\alpha_1$ if
$n\in\{n_k,\ldots,n_{k+1}-1\}$ and $k$ is an even integer.

Let $b_n(\beta_1,\beta_2)=\beta_1$ if $n\in\{1,\ldots,n_1-1\}$,
$b_n(\beta_1,\beta_2)=\beta_2$ if $n\in\{n_k,\ldots,n_{k+1}-1\}$ and $k$ is not an even integer;
$b_n(\beta_1,\beta_2)=\beta_1$ if $n\in\{n_k,\ldots,n_{k+1}-1\}$ and $k$ is an even integer.
This is possible since for fixed $p$ following relations hold:
$$\lim\limits_{k\to\infty}\frac{\sum\limits^{k}_{i=p}[\lambda s_i]}{\sum\limits^{k}_{i=1}s_i}=
\lim\limits_{k\to\infty}\left( \frac{\sum\limits^{k}_{i=1}[\lambda
    s_i]}{\sum\limits^{k}_{i=1}s_i}-
  \frac{\sum\limits^{p-1}_{i=1}[\lambda
    s_i]}{\sum\limits^{k}_{i=1}s_i} \right)=\lambda-0=\lambda.$$
Denote
$x_n=\frac{\sum\limits^{n}_{i=1}[a_i(\alpha_1,\alpha_2)s_i]}{\sum\limits^{k}_{i=1}s_i}$,
$y_n=\frac{\sum\limits^{n}_{i=1}[b_i(\beta_1,\beta_2)s_i]}{\sum\limits^{k}_{i=1}s_i}.$
Suppose the limits $\lim\limits_{k\to\infty}x_n$ and
$\lim\limits_{k\to\infty}y_n$ exist. Let
$\delta<\min(\alpha_2-\alpha_1-2\varepsilon,\beta_2-\beta_1-2\varepsilon)$.
From the Cauchy criterion, it follows that there are $N_1,N_2\in N$
such that for any $a,b>N_1$ and $c,d>N_2$ the inequalities
$|x_a-x_b|<\varepsilon$ and $|y_c-y_d|<\varepsilon$ hold. For a
sufficiently large $k$ we have $n_k>N_j,~j\in\{1,2\}$, whence
$$|x_{n_{k+1}}-x_{n_k}|=\alpha_2-\alpha_1-2\varepsilon>\delta
\quad  \text{and}\quad
  |y_{n_{k+1}}-y_{n_k}|=\beta_2-\beta_1-2\varepsilon>\delta.$$
This contradiction proves the lemma.
\end{proof}

\section{Properties of the set $\Theta_2$}
Let
 $(s_k)$ be a sequence of positive integers such that
$$ \lim\limits_{k\to\infty}s_k=\infty,\quad
\lim\limits_{k\to\infty}\displaystyle\frac{s_{k+1}}{\sum\limits^{k}_{i=1}s_i}=0,\quad
\lim\limits_{k\to\infty}\displaystyle\frac{k}{\sum\limits^{k}_{i=1}s_i}=0.
$$ Let
$
\|\tau_{in}\|
$
be a matrix of dimension $(4\times\infty)$.
Consider the following form of a real number $x\in[0;1]$:
$$
  \hat{x}=\Delta^4_{\underbrace{\underbrace{0\ldots0}_{[\tau_{01}s_1]}
                          \underbrace{1\ldots1}_{[\tau_{11}s_1]}
                          \underbrace{2\ldots2}_{[\tau_{21}s_1]}
                          \underbrace{3\ldots3}_{[\tau_{31}s_1]}}_{\text{1st block}}\ldots
              \underbrace{\underbrace{0\ldots0}_{[\tau_{0k}s_k]}
                          \underbrace{1\ldots1}_{[\tau_{1k}s_k]}
                          \underbrace{2\ldots2}_{[\tau_{2k}s_k]}
                          \underbrace{3\ldots3}_{[\tau_{3k}s_k]}}_{\text{kth block}}\ldots}.
$$

In paper \cite{My3} we proved the following three theorems.
\begin{theorem}\label{teo1}
If
$
\|\tau_{in}\|
$
is a matrix of dimension
$(4\times\infty)$ such that for all $n\in N$ the following conditions hold:
$\tau_{0n}+\tau_{1n}+\tau_{2n}+\tau_{3n}=1$, $\tau_{1n}+2\tau_{2n}+3\tau_{3n}=\theta$, then
$$\lim\limits_{n\to\infty}r_n(\hat{x})=\theta.$$
\end{theorem}

\begin{theorem}\label{teo2}
  If $ \|\tau_{in}\| $ is a stochastic matrix of dimension
  $(4\times\infty)$ and for any fixed $j\in\{0,1,2,3\}$, $
  \lim\limits_{n\to\infty}\tau_{jn}=\lambda, $ then
$$
\nu_j(\hat{x})=\lambda.
$$
\end{theorem}

\begin{theorem}\label{teo3}
  Let $(s^{(1)}_k)$, $(s^{(2)}_k)$ be sequences of positive numbers
  such that $\lim\limits_{k\to\infty}s^{(r)}_k=\infty$, $r\in\{1,2\}$
  and $\|p^{(1)}\|=\|p^{(1)}_{in}\|$, $\|p^{(2)}\|=\|p^{(2)}_{in}\|$
  be stochastic matrices of dimension $(4\times\infty)$.  Let
$$
x(\|p^{(r)}\|;\|s^{(j)}_k\|)= \Delta^4_
  {\underbrace{\underbrace{0\ldots0}_{[p^{(r)}_{01}s^{(j)}_1]}
               \underbrace{1\ldots1}_{[p^{(r)}_{11}s^{(j)}_1]}
               \underbrace{2\ldots2}_{[p^{(r)}_{21}s^{(j)}_1]}
               \underbrace{3\ldots3}_{[p^{(r)}_{31}s^{(j)}_1]}}_{\text{1st block}}\ldots
   \underbrace{\underbrace{0\ldots0}_{[p^{(r)}_{0k}s^{(j)}_k]}
               \underbrace{1\ldots1}_{[p^{(r)}_{1k}s^{(j)}_k]}
               \underbrace{2\ldots2}_{[p^{(r)}_{2k}s^{(j)}_k]}
               \underbrace{3\ldots3}_{[p^{(r)}_{3k}s^{(j)}_k]}}_{\text{kth block}}\ldots}.
$$

If
$
\lim\limits_{k\to\infty}|s^{(1)}_k-s^{(2)}_k|=\infty
$,
then
$
x(\|p^{(1)}\|;\|s^{(1)}_k\|)\neq x(\|p^{(2)}\|;\|s^{(2)}_k\|)
$.

If
$
\varlimsup\limits_{n\to\infty}\sum\limits^{3}_{i=0}|p^{(1)}_{in}-p^{(2)}_{in}|>0
$,
then
$
x(\|p^{(1)}\|;\|s^{(1)}_k\|)\neq x(\|p^{(2)}\|;\|s^{(2)}_k\|)
$.
\end{theorem}

\begin{theorem}
  If $\theta\in(0;3)$, then the set $\Theta_2$ is an everywhere dense,
  continuum set of zero Lebesgue measure.
\end{theorem}
\begin{proof}
  The well-known Borel's theorem states that
  $\nu_0=\nu_1=\nu_2=\nu_3=\dfrac{1}{4}$ for almost all in the sense
  of Lebesgue measure numbers of $[0;1]$. From this fact it follows
  that Lebesgue measure of the set $\Theta_2$ is equal to zero.

  \emph{Continuality.} Construct a continuum subset of $\Theta_2$ such
  that frequency of the digit $0$ exists for all elements (similarly
  we can construct a continuum subset of $\Theta_2$ such that the
  frequency of a fixed digit ``$i$'', $i\in\{1,2,3\}$, exists for all
  elements). Let $s_k=k$ and $\overline{p}=(p_0,p_1,p_2,p_3)$,
  $\overline{q}=(q_0,q_1,q_2,q_3)$ be stochastic vectors such that
  $p_0=q_0$, $p_1+2p_2+3p_3=q_1+2q_2+3q_3=\theta$, $p_1\neq q_1$, then
  $ \lim\limits_{k\to\infty}s_k=\infty $, $
  \lim\limits_{k\to\infty}\displaystyle\frac{k}{\sum\limits^{k}_{i=1}s_i}=0
  $, $
  \lim\limits_{k\to\infty}\displaystyle\frac{s_{k+1}}{\sum\limits^{k}_{i=1}s_i}=0.
  $ From Lemma \ref{teo4}, where $\alpha_1=\beta_1=p_1$,
  $\alpha_2=\beta_2=q_1$, it follows that there exists a sequence
  $a_n(p_1,q_1)$ such that $a_n(p_1,q_1)=p_1$ or $a_n(p_1,q_1)=q_1$
  for any $n\in N$ and
  $\lim\limits_{k\to\infty}\frac{\sum\limits^{k}_{i=1}[a_i(p_1,q_1)s_i]}{\sum\limits^{k}_{i=1}s_i}$
  does not exist.  Denote $\tau_{0k}=p_0$,
  $\tau_{1k}=a_k(p_1,q_1)$. Using the system
$$\left\{
\begin{aligned}
   &\tau_{2k}+\tau_{3k}=1-p_0-a_k(p_1,q_1),\\
   &2\tau_{2k}+3\tau_{3k}v^{(n)}_3=\theta-a_k(p_1,q_1)\\
\end{aligned}
\right.$$ we calculate $\tau_{2k},~\tau_{3k}$. Namely,
$\tau_{3k}=\theta+a_k(p_1,q_1)-2+2p_0$
$\tau_{2k}=3-3p_0-\theta-2a_k(p_1,q_1)$. It is evident that
$\tau_{ik}$, where $i\in\{2,3\}$ is equal to $p_i$ or $q_i$ if
$a_k(p_1,q_1)$ is equal to $p_1$ or $q_1$, respectively.
From Theorem \ref{teo2} it follows that $\nu_0(x)=p_0$. Since
$\lim\limits_{k\to\infty}\frac{\sum\limits^{k}_{i=1}[a_i(p_1,q_1)s_i]}{\sum\limits^{k}_{i=1}s_i}$
does not exist we obtain that the frequency $\nu_1(x)$ does not exist
either.  From Theorem \ref{teo3} it follows that different numbers $x$
constructed as specified above correspond to different pairs of
vectors $\overline{p}$ and $\overline{q}$ with relevant
properties. Since the set of such pairs is a continuum, we see that the
set $\Theta_2$ is a continuum.

\emph{Everywhere density.} Since the condition
$\lim\limits_{k\to\infty}r_k(x)=\theta$ does not depend on an
arbitrary finite group of first symbols and for any interval $[a;b]$
there exists a cylinder
$[\Delta^4_{\gamma_1\gamma_2\ldots\gamma_r(0)};\Delta^4_{\gamma_1\gamma_2\ldots\gamma_r(3)}]$
completely contained in it, we see that $\Theta_2$ is an everywhere dense
set.
\end{proof}

\begin{theorem}\label{teoanalog}
  If $\theta\in(0;3)$, then the Hausdorff--Besicovitch dimension
  $\alpha_0(\Theta_2)$ of the set $\Theta_2$ is positive,
  i.e., $\alpha_0(\Theta_2)>0$.
\end{theorem}

\begin{proof}
   Let $(\varepsilon_i)$ be an arbitrary sequence of zeros and ones, vectors $(p_0,p_1,p_2,p_3)$ and $(p_0,q_1,q_2,q_3)$ be stochastic vectors
  such that $p_0>0$, $p_1\neq q_1$,
  $p_1+2p_2+3p_3=\theta=q_1+2q_2+3q_3$, $x_i=[p_0k(i+1)]-[p_0ki]-r_i$,
  $r_i=\left\{
       \begin{aligned}
       &0,\quad \text{if}\ \ \varepsilon_i=1\\
       &1,\quad \text{if}\ \  \varepsilon_i=0
       \end{aligned}
     \right. ,  $
$t_i=[p_3k(i+1)]-[p_3ki]$. 

Consider the system
\begin{equation}
  \label{2*}
  \left\{
    \begin{aligned}
       &x_i+y_i+z_i+t_i=k-1,\\
       &y_i+2z_i+3t_i=[\theta k(i+1)]-[\theta ki],
     \end{aligned}
  \right.
\end{equation}
whence $z_i=[\theta k(i+1)]-[\theta ki]-k+1+x_i-2y_i,$
$y_i=2(k-1)-([\theta k(i+1)]-[\theta ki])-2x_i+t_i.$ We obtain that
$$
\begin{aligned}
\displaystyle\frac{z_i}{k}&=\displaystyle\frac{[\{\theta
ki\}+\theta k]}{k}-1+
                            \displaystyle\frac{1}{k}+
                            \displaystyle\frac{[\{p_0ki\}+p_0k]}{k}-
                            \displaystyle\frac{2[\{p_3ki\}+p_3k]}{k}  \\
                           &=\theta-1+p_0-2p_3+
                            \displaystyle\frac{[\{\theta ki\}+\{\theta k\}]}{k}+
                            \displaystyle\frac{1}{k}+
                            \displaystyle\frac{[\{p_0ki\}+\{p_0k\}]}{k}-
                            \displaystyle\frac{2[\{p_3ki\}+\{p_3k\}]}{k}.
\end{aligned}
 $$

Since $\theta-1+p_0-2p_3=p_2$ and for a sufficiently large $k$ we
have $z_i\in N$ for all $i\in N$. In the same way,
$$
\begin{aligned}
\displaystyle\frac{y_i}{k}&=2-\displaystyle\frac{2}{k}-
                            \displaystyle\frac{[\{\theta ki\}+\theta k]}{k}-
                            \displaystyle\frac{2[\{p_0ki\}+p_0k]}{k}+
                            \displaystyle\frac{r_i}{k}+
                            \displaystyle\frac{[\{p_3ki\}+p_3k]}{k}\\
                           &=2-\theta-2p_0+3p_3-
                            \displaystyle\frac{2}{k}-
                            \displaystyle\frac{[\{\theta ki\}+\{\theta k\}]}{k}-
                            \displaystyle\frac{2[\{p_0ki\}+\{p_0k\}]}{k}
                            \\
                            &  +
                            \displaystyle\frac{r_i}{k}+
                            \displaystyle\frac{[\{p_3ki\}+\{p_3k\}]}{k}.
\end{aligned}
$$

 Since $2-\theta-2p_0+3p_3=p_1$ and for a sufficiently large $k$
we obtain $y_i\in N$ for all $i\in N$. Similarly we prove that for a
sufficiently large $k\in N$ all solutions of the system
\begin{equation}
  \label{**}
  \left\{
    \begin{aligned}
       &x_i+y_i+z_i+t_i=k-1,\\
       &y_i+2z_i+3t_i=[\theta k(i+1)]-\theta ki],\\
       &x_i=[p_0k(i+1)]-[p_0ki]-r_i,\\
       &t_i=[q_3k(i+1)]-[q_3ki],
     \end{aligned}
  \right.
\end{equation}
are positive integers for all $i\in N$.

Let $k$ be a sufficiently large positive integer. Let all solutions of
systems \eqref{2*} and \eqref{**} be positive integers for arbitrary
sequence of zeros and ones $(\varepsilon_i)$, $i\in N$.  Construct the
number $x(\varepsilon_i)$ as follows:
 $$
x(\varepsilon_i)=\Delta^4_{\underbrace{\varepsilon_1
    \underbrace{0\ldots0}_{x_1} \underbrace{1\ldots1}_{y_1}
    \underbrace{2\ldots2}_{z_1} \underbrace{3\ldots3}_{t_1}}_{\text{k
      symbols}}\ldots \underbrace{\varepsilon_j
    \underbrace{0\ldots0}_{x_j} \underbrace{1\ldots1}_{y_j}
    \underbrace{2\ldots2}_{z_j} \underbrace{3\ldots3}_{t_j}}_{\text{k
      symbols}}\ldots}
$$
Without loss of generality put $p_3>q_3$, let $\delta>0$ be such that
$p_3-\delta>q_3-\delta$. Let $r_1$ be a positive integer such that
$(x_i,y_i,z_i,t_i)$ is a solution of system \eqref{2*} for any
$j\in\{1,2,\ldots,r_1\}$ and
$$
    \displaystyle\frac{N_3(x,kr_1)}{kr_1}=
    \displaystyle\frac{\sum\limits^{r_1}_{i=1}t_i}{kr_1}=
    \displaystyle\frac{[p_3k(r_1+1)]}{kr_1}>p_3-\delta,
$$
this is possible since the last value tends to $p_3$ as $r_1\to\infty.$

Let $r_1<r_2$ be a positive integer such that $(x_j,y_j,z_j,t_j)$ is a
solution of system \eqref{**} for any $j\in\{r_1+1,\ldots,r_2\}$ and
$$
\displaystyle\frac{N_3(x,kr_2)}{kr_2}=\displaystyle\frac{\sum\limits^{r_2}_{i=1}t_i}{kr_2}=\displaystyle\frac{[p_3k(r_1+1)]-[q_3k(r_1+1)]+[q_3k(r_2+1)]}{kr_2}<q_3+\delta,
$$
this is possible since the last value tends to $q_3$ as $r_2\to\infty.$

Let $r_2<r_3$ be a positive integer such that $(x_j,y_j,z_j,t_j)$ is a
solution of system \eqref{2*} for any $j\in\{r_2+1,\ldots,r_3\}$ and
$$
\begin{aligned}
& \displaystyle\frac{N_3(x,kr_3)}{kr_3}\\
& =3\displaystyle\frac{[p_3k(r_1+1)]-[q_3k(r_1+1)]+[q_3k(r_2+1)]-[p_3k(r_2+1)]+[p_3k(r_3+1)]}{kr_3}>\\
&>p_3-\delta,
\end{aligned}
$$
this is possible since the last value tends to $p_3$ as $r_k\to\infty.$ And so on.

We obtain that
$\left|\displaystyle\frac{N_3(x,kr_i)}{kr_i}-\displaystyle\frac{N_3(x,kr_{i+1})}{kr_{i+1}}\right|>p_3-q_3-2\delta$
for all $i\in N$. Assume that
$\lim\limits_{i\to\infty}\displaystyle\frac{N_3(x,kr_i)}{kr_i}$
exists. Hence, we have a contradiction with Cauchy's criterion.
Thus,
$\lim\limits_{i\to\infty}\displaystyle\frac{N_3(x,kr_i)}{kr_i}$
does not exist, i.e., the frequency $\nu_3(x(\varepsilon_i))$ does not
exist. On the other hand, if $kj\leqslant n\leqslant
k(j+1)$ then
 $$
 \begin{aligned}
\displaystyle\frac{N_0(x(\varepsilon_i),n)}{n}&\geqslant
                   \displaystyle\frac{\sum\limits^{j}_{i=1}[p_0k(i+1)]-[p_0ki]}{k(j+1)}=
                   \displaystyle\frac{[p_0k(j+1)]-[p_0k]}{k(j+1)}\\
                   &=p_0-\displaystyle\frac{\{p_0k(j+1)\}-[p_0k]}{k(j+1)}\to p_0,\quad \text{as}\quad j\to\infty,
\\
\displaystyle\frac{N_0(x(\varepsilon_i),n)}{n}&\leqslant
                   \displaystyle\frac{\sum\limits^{j+1}_{i=1}[p_0k(i+1)]-[p_0ki]}{kj}=
                   \displaystyle\frac{\{p_0k(j+2)\}-[p_0k]}{kj}\\
                   &=p_0\displaystyle\frac{j+2}{j}-\frac{\{p_0k(j+2)\}-[p_0k]}{kj}\to p_0,\quad \text{as}\quad j\to\infty,
\end{aligned}
$$
hence, $\nu_0(x(\varepsilon_i))=p_0$. Also
 $$
\begin{aligned}
 r_n(x(\varepsilon_i))&\geqslant
                   \displaystyle\frac{\sum\limits^{j}_{i=1}[\theta k(i+1)]-[\theta ki]}{k(j+1)}=
                   \displaystyle\frac{[\theta k(j+1)]-[\theta k]}{k(j+1)}\\
                   &=\theta-\displaystyle\frac{\{\theta k(j+1)\}-[\theta k]}{k(j+1)}\to \theta,\quad \text{as}\quad j\to\infty,
\\
 r_n(x(\varepsilon_i))&\leqslant
                   \displaystyle\frac{\sum\limits^{j+1}_{i=1}[\theta k(i+1)]-[\theta ki]}{kj}=
                   \displaystyle\frac{\{\theta k(j+2)\}-[\theta k]}{kj}\\
                   &=\theta\displaystyle\frac{j+2}{j}-\frac{\{\theta k(j+2)\}-[\theta k]}{kj}\to \theta,\quad \text{as}\quad j\to\infty,
\end{aligned}
$$
hence, $\lim\limits_{n\to\infty}r_n(x(\varepsilon_i))=\theta$,
and from Theorem \ref{teo3}, the frequencies $\nu_1(x(\varepsilon_i))$
and $\nu_2(x(\varepsilon_i))$ do not exist.

Thus, $x(\varepsilon_i)\in \Theta_2$.

Selecting an arbitrary quantity of (not necessarily consecutive)
blocks of number $x(\varepsilon_i)$ and changing the order of digits
(except for $\varepsilon_{i}$) inside each block we get either the
``old" number $x(\varepsilon_i)$, or a new number
$\widetilde{x}(\varepsilon_i)$. These numbers belong to $\Theta_2$
since $N_l(x(\varepsilon_i),kr)=N_l(\widetilde{x}(\varepsilon_i),kr) $
for any $r \in N$ and $l \in \{0,1,2,3\}$. Denote by
$C(x(\varepsilon_i))$ the set of numbers
$\widetilde{x}(\varepsilon_i)$ obtained from $x(\varepsilon_i)$ by
choosing an arbitrary number of blocks and changing the digit order
inside them. It is obvious that the set is a continuum. Denote by
$C_1$ a union of the sets $C(x(\varepsilon_i))$ with respect to all
possible sequences $(\varepsilon_i)$ and show that
$\alpha_0(C_1)=\frac{1}{2k}$.

To calculate the Hausdorff--Besicovitch dimension it is sufficient to
use covering of $4$-adic cylinders. Consider a covering of the set
$C_1$ by cylinders of the same rank $m$. the $\alpha$-volume of the
covering is equal to
\[
R^{\alpha}_m=
             \left\{
              \begin{aligned}
                 & 2^{t-1}(4^{-(kt-j)})^{\alpha},\quad \text{if}\quad   m=kt-j,~j\in\{1,\ldots,k-1\},\\
                 & 2^t(4^{-kt})^{\alpha},\quad\text{if}\quad m=kt.
              \end{aligned}
              \right.
\]
It is clear that $R^{\alpha}_{kt-1}<R^{\alpha}_{kt-j}$,
$j\in\{2,\ldots,k-1\}$ hence, consider an $\alpha$-covering of the set $C_1$
with cylinders of rank $n=kt-1$.

The Hausdorff's box--counting $\alpha$-measure of the set $C_1$ is equal to
\[\widehat{H}_{\alpha}(C_1)=\varliminf\limits_{t\to\infty}\frac{2^{t-1}}{4^{(kt-1)\alpha}}=
2^{2\alpha-1}\varliminf\limits_{t\to\infty}2^{t(1-2k\alpha)}.\]
Whence,
\[\widehat{H}_{\alpha}(C_1)=\left\{
  \begin{aligned} & 0,\quad \text{if}\ \ \alpha>\frac{1}{2k},\\ &
\infty,\quad\text{if}\ \ \alpha<\frac{1}{2k}.
\end{aligned} \right.\] Therefore, box--counting dimension of the set
$C_1$ is equal to $\alpha=\frac{1}{2k}$. Let us show that
$\alpha_0(C_1)=\frac{1}{2k}$. Consider an arbitrary finite covering of
the set $C_1$ by $4$-adic cylinders $\{v_j\}$, $j\in\{1,\ldots,l\}$,
and prove that if $\alpha=\frac{1}{2k}$ then the preceding rank covering
is not improvable. Let $u_i$ be one of cylinders of the covering. Then
$|u_j|=3^{-n}$ for some $n\in N$. Let $n=kp-r,~r\in\{0,\ldots,k-1\}$,
then  $\alpha$-volume of covering of the set $C_1\cap\Delta_j$ by cylinders
of rank $kp-r+m$ is equal to
$$
R^{\alpha}_{kp-r+m}(C_1\cap\Delta_j)=\left\{
                                \begin{aligned}
                 & 2^{l-1}(4^{-(kp-r+kl-j)})^{\alpha},\quad \text{if}\ \ m=kl+r-j,~j\in\{1,\ldots,k-1\},\\
                 & 2^k(4^{-(kp-r+kl)})^{\alpha},\quad \text{if}\ \ m=kl+r,~l\in N.
                                \end{aligned}
                                     \right.
$$
Let us show that $R^{\alpha}_{kp-r+m}(C_1\cap v_j)\leqslant
v_n=(4^{-(kp-r)})^{\alpha}$. It is obvious that
$R^{\alpha}_{k(p+l)-1}<R^{\alpha}_{k(p+l)-j}$, if
$j\in\{2,\ldots,k\}$. Consider an $\alpha$-covering $K\cap\Delta_j$ by
cylinders of rank $k(p+l)-1$. Its volume is equal to
$$2^{l-1}(4^{-(kp+3l-1)})^{\alpha}.$$
Since $\left(\displaystyle\frac{2}{4^{kl}}\right)^l=1$ and $\displaystyle\frac{4^{(1-r)\alpha}}{2}<1$ we obtain if $\alpha=\displaystyle\frac{1}{2k}$ then\\
$
2^{l-1}(4^{-(kp+3l-1)})^{\alpha}=
\left(4^{-(kp-r)}\right)^{\alpha}\displaystyle\frac{4^{(1-r)\alpha}}{2}
\left(\displaystyle\frac{2}{4^{kl}}\right)^l\leqslant
\left(4^{-(kp-r)}\right)^{\alpha}.
$
Hence if $\alpha=\displaystyle\frac{1}{2k}$ we have
$\widehat{H}_{\alpha}(C_1)=H_{\alpha}(C_1)=\displaystyle\frac{1}{2k}$
and we see that the Hausdorff--Besicovitch dimension of the set $C_1$  is
equal to the box--counting dimension of the set $C_1\subset \Theta_2$,
thus
$\alpha_0(\Theta_2)\geqslant\alpha_0(C_1)=\displaystyle\frac{1}{2k}>0$.\\
\end{proof}

\section{The set $\Theta_3$}

\begin{theorem}
  If $\theta\in(0;3)$, then the set $\Theta_3$ is an everywhere dense,
  continuum set of zero Lebesgue measure.
\end{theorem}
\begin{proof} \emph{Lebesgue measure.}  Since almost all (in the sense
  of Lebesgue measure) numbers of the interval $[0;1]$ are normal,
  i.e., $\nu_0=\nu_1=\nu_2=\nu_3=\dfrac{1}{4}$ \cite{Bor1} we see that
  Lebesgue measure of the set $\Theta_3$ is equal to zero.

  \emph{Continuality.} Let $s_k=k$, $p_0>q_0>0$, $p_1>q_1>0$. Suppose that
  all solutions of the system
$$
\left\{
 \begin{aligned}
     &x+y+z+t=1,\\
     &y+2z+3t=\theta,\\
     &x=p_0\vee q_0,\\
     &y=p_1\vee q_1
 \end{aligned}
\right.
$$
are positive.

It is obvious that $ \lim\limits_{k\to\infty}s_k=\infty $, $
\lim\limits_{k\to\infty}\displaystyle\frac{k}{\sum\limits^{k}_{i=1}s_i}=0
$ $
\lim\limits_{k\to\infty}\displaystyle\frac{s_{k+1}}{\sum\limits^{k}_{i=1}s_i}=0
$.  From Lemma \ref{teo4}, where $\alpha_1=p_0$, $\alpha_2=q_0$,
$\beta_1=p_1$, $\beta_2=q_1$ it follows that there exist sequences
$a_n(p_0,q_0)=p_0$ or $a_n(p_0,q_0)=q_0$ and $b_n(p_1,q_1)=p_1$ or
$b_n(p_1,q_1)=q_1$ such that for all $n\in N$ the limits
$$
\lim\limits_{k\to\infty}\displaystyle\frac{\sum\limits^{k}_{i=1}[a_i(p_0,q_0)s_i]}
                                          {\sum\limits^{k}_{i=1}s_i}
\quad \text{and}\quad
\lim\limits_{k\to\infty}\displaystyle\frac{\sum\limits^{k}_{i=1}[b_i(p_1,q_1)s_i]}
                                          {\sum\limits^{k}_{i=1}s_i}
$$
do not exist.

Denote
$\tau_{0k}=a_k(p_0,q_0)$, $\tau_{1k}=b_k(p_1,q_1)$.  From following system
$$
\left\{
\begin{aligned}
 &\tau_{0k}+\tau_{1k}+\tau_{2k}+\tau_{3k}=1,\\
 &\tau_{1k}+2\tau_{2k}+3\tau_{3k}=\theta
\end{aligned}
\right.
$$
we obtain $\tau_{2k}$, $\tau_{3k}$, i.e.,
$\tau_{3k}=\theta-2+2\tau_{0k}+\tau_{1k}$,
$\tau_{2k}=3-\theta-3\tau_{0k}-\tau_{1k}.$

Since the limits
$$
\lim\limits_{k\to\infty}\displaystyle\frac{N_0(x,\sum\limits^{k}_{i=1}s_i)}
                                          {\sum\limits^{k}_{i=1}s_i}=
\lim\limits_{k\to\infty}\displaystyle\frac{\sum\limits^{k}_{i=1}[a_i(p_0,q_0)s_i]}
                                          {\sum\limits^{k}_{i=1}s_i}
$$
and
$$
\lim\limits_{k\to\infty}\displaystyle\frac{N_1(x,\sum\limits^{k}_{i=1}s_i)}
                                          {\sum\limits^{k}_{i=1}s_i}=
\lim\limits_{k\to\infty}\displaystyle\frac{\sum\limits^{k}_{i=1}[b_i(p_1,q_1)s_i]}
                                          {\sum\limits^{k}_{i=1}s_i}
$$
do not exist, the frequencies $\nu_0(x)$ and $\nu_1(x)$ do not exist
either. Then from Theorem \ref{teo1} and from Theorem \ref{teo3} it
follows that $\lim\limits_{n\to\infty}r_n(x)=\theta$ and $\nu_2(x)$,
$\nu_3(x)$ do not exist.

From Theorem \ref{teo3} it follows that different numbers constructed
as indicated above correspond to different pairs $(p_0,q_0)$ and
$(p_1,q_1)$. Since the set of such pairs is a continuum, we obtain that set
$\Theta_3$ is a continuum.

\emph{Everywhere density.} Since the condition
$\lim\limits_{k\to\infty}r_k(x)=\theta$ does not depend on an
arbitrary finite group of first symbols, and for any interval
$[a;b]\subset[0;1]$ there exists a cylinder
$[\Delta^4_{\gamma_1\gamma_2\ldots\gamma_r(0)};\Delta^4_{\gamma_1\gamma_2\ldots\gamma_r(3)}]$
completely contained in it, we see that $\Theta_3$ is an everywhere dense
set.
\end{proof}

\begin{theorem}
  If $\theta\in(0;3)$, then the Hausdorff--Besicovitch dimension
  $\alpha_0(\Theta_3)$ of the set $\Theta_3$ is positive,
  i.e., $\alpha_0(\Theta_3)>0$.
\end{theorem}
\begin{proof}
  Let $(\varepsilon_i)$ be an arbitrary sequence of zeros and ones,
  let $p^{(1)}_0>p^{(2)}_0>0$ and $p^{(1)}_1>p^{(2)}_1>0$, let
  solutions of the system
$$
\left\{
\begin{aligned}
       &x+y+z+t=1,\\
       &y+2z+3t=\theta,\\
       &x=p^{(1)}_0\vee p^{(2)}_0,\\
       &y=p^{(1)}_1\vee p^{(2)}_1
\end{aligned}
\right.
$$
be positive.

Denote
$$r_i=
\left\{
\begin{aligned}
 &0,~\text{if}~\varepsilon_i=1,\\
 &1,~\text{if}~\varepsilon_i=0,
\end{aligned}
\right.\qquad
\widetilde{r_i}=
\left\{
\begin{aligned}
 &0,~\text{if}~\varepsilon_i=0,\\
 &1,~\text{if}~\varepsilon_i=1.
\end{aligned}
\right.$$

Similarly to the proof of Theorem \ref{teoanalog}, we show
existence of a sufficiently large positive integer $k$ such that all
solutions of the  systems
  \begin{equation}
  \label{***}
  \left\{
    \begin{aligned}
       &x_i+y_i+z_i+t_i=k+1,\\
       &y_i+2z_i+3t_i=[\theta k(i+1)]-[\theta ki],\\
       &x_i=[p^{(1)}_0k(i+1)]-[p^{(1)}_0ki]-r_i,\\
       &t_i=[p^{(1)}_1k(i+1)]-[p^{(1)}_1ki]-\widetilde{r_i},
     \end{aligned}
\right.
\end{equation}
\begin{equation}
  \label{****}
  \left\{
    \begin{aligned}
       &x_i+y_i+z_i+t_i=k+1,\\
       &y_i+2z_i+3t_i=[\theta k(i+1)]-[\theta ki],\\
       &x_i=[p^{(2)}_0k(i+1)]-[p^{(2)}_0ki]-r_i,\\
       &t_i=[p^{(2)}_1k(i+1)]-[p^{(2)}_1ki]-\widetilde{r_i}
     \end{aligned}
\right.
\end{equation}
are positive for all $i\in N$.

Let $(\varepsilon_i)$ be a fixed sequence of zeros and ones. Construct a number $x(\varepsilon_i)$ as follows:\\
$$
x(\varepsilon_i)=\Delta^4_{\underbrace{\varepsilon_1
                                       \underbrace{0\ldots0}_{x_1}
                                       \underbrace{1\ldots1}_{y_1}
                                       \underbrace{2\ldots2}_{z_1}
                                       \underbrace{3\ldots3}_{t_1}}_{\text{k symbols}}\ldots
                           \underbrace{\varepsilon_j
                                       \underbrace{0\ldots0}_{x_j}
                                       \underbrace{1\ldots1}_{y_j}
                                       \underbrace{2\ldots2}_{z_j}
                                       \underbrace{3\ldots3}_{t_j}}_{\text{k symbols}}\ldots}
$$
Let $\delta>0$ be such that $p^{(1)}_0-\delta>p^{(2)}_0+\delta$, $p^{(1)}_1-\delta>p^{(2)}_1+\delta$.

Let $g_1$ be a positive integer such that $(x_j,y_j,z_j,t_j)$ is
a solution of system (4) for any $j\in\{1,2,\ldots,g_1\}$ and
$\displaystyle\frac{N_0(x,kg_1)}{kg_1}>p^{(1)}_0-\delta$,
$\displaystyle\frac{N_1(x,kg_1)}{kg_1}>p^{(1)}_1-\delta$.

Let $g_2$ be a positive integer such that $(x_j,y_j,z_j,t_j)$ is
a solution of system (5) for all $j\in\{g_1+1,g_1+2,\ldots,g_2\}$ and
$\displaystyle\frac{N_0(x,kg_2)}{kg_2}<p^{(2)}_0+\delta$,
$\displaystyle\frac{N_1(x,kg_2)}{kg_2}<p^{(2)}_1+\delta$.

Let $g_3$ be a positive integer such that $(x_j,y_j,z_j,t_j)$ is
a solution of system (4) for any $j\in\{g_2+1,g_2+2,\ldots,g_3\}$ and
$\displaystyle\frac{N_0(x,kg_3)}{kg_3}>p^{(1)}_0-\delta$,
$\displaystyle\frac{N_1(x,kg_3)}{kg_3}>p^{(1)}_1-\delta$.  And so on.

Since 
$$\left|\displaystyle\frac{N_a(x,kg_{j+1})}{kg_{j+1}}-
  \displaystyle\frac{N_a(x,kg_j)}{kg_j}
\right|>p^{(1)}_a-p^{(2)}_a-2\delta$$ for all $j\in N$,
the limits
$\lim\limits_{j\to\infty}\displaystyle\frac{N_a(x,kg_j)}{kg_j}$,
$a\in\{0,1\}$ do not exist (assuming the converse, we obtain
a contradiction to the Cauchy criterion). Thus, $\nu_0(x(\varepsilon_i))$
and $\nu_1(x(\varepsilon_i))$ do not exist.

Let 
$$
\begin{aligned}
& kj\leqslant n<k(j+1), \\
& r_n\geqslant
       \displaystyle\frac{[\theta k(j+1)]-[\theta k]}{k(j+1)}=
       \theta-\displaystyle\frac{\{\theta k(j+1)\}+[\theta k]}{k(j+1)}\to
       \theta,
\\
& r_n\leqslant
       \displaystyle\frac{[\theta k(j+2)]-[\theta k]}{kj}=
       \theta\cdot\displaystyle\frac{j+2}{j}-\frac{\{\theta k(j+2)\}-[\theta k]}{kj}\to
       \theta \quad (j\to\infty).
\end{aligned}
$$
Hence, $r_n\to\theta$ as $n\to\infty$ and from Theorem \ref{teo3}, it
follows that the frequencies $\nu_2(x(\varepsilon_i))$ and
$\nu_3(x(\varepsilon_i))$ do not exist, i.e., $x(\varepsilon_i)\in
\Theta_3$.

Selecting an arbitrary quantity of (not necessarily consecutive)
blocks of number $x(\varepsilon_i)$ and changing the order of digits
inside each block (except for $\varepsilon_{i}$) we obtain either the
``old" number $x(\varepsilon_i)$, or a new number
$\widetilde{x}(\varepsilon_i)$. These numbers are contained in
$\Theta_3$ since
$N_l(x(\varepsilon_i),kr)=N_l(\widetilde{x}(\varepsilon_i),kr)$, for
any $r\in N$ and $l\in\{0,1,2,3\}$. Denote by $C(x(\varepsilon_i))$
the set of numbers $\widetilde{x}(\varepsilon_i)$ obtained from
$x(\varepsilon_i)$ by choosing an arbitrary number of blocks and changing
digit order inside them. It is evident that the set is a
continuum. Denote by $C_1$ a union of the sets $C(x(\varepsilon_i))$ of all
possible sequences $(\varepsilon_i)$ and show that
$\alpha_0(C_1)=\frac{1}{2k}$. Similarly to the proof of
Theorem~\ref{teoanalog}, we show that
$\alpha_0(C_1)=\displaystyle\frac{1}{2k}$.

Thus, $\alpha_0(\Theta_3)\geqslant\alpha_0(C_1)>0$.
\end{proof}


\begin{thebibliography}{100000}


\bibitem{AlPrTor}
{S. Albeverio, M. Pratsiovytyi, G. Torbin}, 
\textit{Singular probability distributions and fractal properties of sets of real numbers
defined by the asymptotic frequencies of their s-adic digits},    
{Ukrain. Mat. Zh.}    
\textbf{57} 
{(2005)},   
{no.~9},    
{1163--1170; English transl.}   
{Ukrainian Math.~J.}    
\textbf{57} 
{(2005)},   
{no.~9},    
{1361--1370}.   



\bibitem{AlPrTor1}
{S. Albeverio, M. Pratsiovytyi, G. Torbin}, 
\textit{Topological and fractal properties of real numbers which are not normal},   
{Bull. Sci. Math.}  
\textbf{129}    
{(2005)},   
{no.~8},    
{615--630}. 

\bibitem{Besic2}
{A. S. Besicovitch}, 
\textit{Sets of fractional dimension. {\rm 2}: On the sum of digits of real numbers represented in the dyadic system},    
{Math. Ann.}    
\textbf{110}    
{(1934)},   
{no.~3},    
{321--330}. 

\bibitem{Bill}
{P. Billingsley},   
\textit{Ergodyc Theory and Information},    
{John Wiley and Sons Inc.}, 
{New York---London---Sydney},     
{1965}.     


\bibitem{Bor1}
{{\'E}. Borel}, 
\textit{Les probabilit\'{e}s d\'{e}nombrables et leurs applications arithm\'{e}tiques}, 
{Rend. Circ. Mat.}  
\textbf{27} 
{(1909)},   
{no.~1},    
{247--271}. 


\bibitem{Egg1}
{H. G. Eggleston},   
\textit{The fractional dimension of a set defined by decimal properties},   
{Quart. J.~Math.}   
\textbf{20} 
{(1949)},   
{no.~4},    
{31--36}.   

\bibitem{Ols3}
{L. Olsen}, 
\textit{Normal and non-normal points of self-similar sets and divergence points of self-similar measures},  
{J.~London Math. Soc.}  
\textbf{2(67)}  
{(2003)},   
{no.~1},    
{103--122}. 

\bibitem{My2}
{M. V. Pratsiovytyi, S. O. Klymchuk}, 
\textit{Linear fractals of Besicovitch--Eggleston type},    
{Scientific journal of Dragomanov NPU. Series 1. Physics and mathematics}  
{(2012)},   
{no.~13(2)},    
{80--92. (Ukrainian)}   


\bibitem{My3}
{M. V. Pratsiovytyi, S. O. Klymchuk}, 
\textit{Topological, metric and fractal properties of sets of real numbers with preassigned mean
of digits of 4-adic representation when their frequencies exist},  
{Scientific journal of Dragomanov NPU. Series 1. Physics and mathematics }  
\textbf{}   
{(2013)},   
{no.~14},   
{217--226. (Ukrainian)} 


\bibitem{My4}
{Pratsiovytyi M.V., Klymchuk S.O., Makarchuk O.P.}, 
\textit{Frequency of ternary digit of number and its asymptotic mean of digits},    
{Ukrain. Mat. Zh.}    
\textbf{66} 
{(2014)},   
{no.~3},    
{302--310. (Ukrainian); English transl.} 
{Ukrainian Math.~J.}    
\textbf{66} 
{(2014)},   
{no.~3},    
{336--346}. 





\bibitem{PrTorb}
{M. V. Pratsiovytyi, G. M. Torbin},   
\textit{Superfractality of the set of numbers having no frequency of $n$-adic digits, and fractal probability distributions},   
{Ukrain. Mat. Zh.}    
\textbf{47} 
{(1995)},   
{no.~7},    
{971--975. (Ukrainian); English transl.} 
{Ukrainian Math.~J.}    
\textbf{47} 
{(1995)},   
{no.~7},    
{1113--1118}. 








\bibitem{Torb}
{G. M. Torbin},  
\textit{Frequency characteristics of normal numbers in different number systems},   
{Fractal Analysis and Related Problems}, 
{Kyiv: Institute of Mathematics, National Academy of Sciences of
Ukraine --  National Pedagogical Dragomanov University, 1998},   
{no.~1},    
{pp.~53--55. (Ukrainian)}   






\end{thebibliography}
\end{document}